\documentclass[10pt]{article}
\usepackage{amsthm}
\usepackage{amsfonts}
\usepackage{amssymb}
\usepackage{amsmath}
\usepackage{nccmath}
\usepackage{mathtools}
\usepackage{mathrsfs}
\usepackage{setspace}
\usepackage{graphicx}
\usepackage{booktabs}
\usepackage{mathpazo}
\usepackage{charter}
\usepackage{mathptmx}
\usepackage{float}
\usepackage[numbers,sort&compress]{natbib}
\usepackage{enumerate}
\usepackage{authblk}

\newtheorem{theorem}{Theorem}[section]

\theoremstyle{definition}
\newtheorem{definition}[theorem]{Definition}
\newtheorem{example}[theorem]{Example}

\newtheorem{corollary}[theorem]{Corollary}

\theoremstyle{remark}
\newtheorem{remark}[theorem]{Remark}

\numberwithin{equation}{section}

\DeclareMathOperator{\sgn}{sign}
\usepackage{amsmath,amssymb,graphicx} 
\usepackage[margin=1in]{geometry} 
\usepackage{enumerate}
\usepackage{authblk}
\usepackage{placeins}
\usepackage{array}
\usepackage{academicons}
\usepackage{xcolor}
\usepackage{booktabs}
\usepackage{svg}
\usepackage[utf8]{inputenc}
\usepackage{tikz}
\usepackage{subcaption}
\usepackage{fancyhdr}
\usepackage{pgfplots}
\usepackage{adjustbox}
\usepackage{hyperref}
\usepackage[T1]{fontenc}
\usepackage{lmodern}
\hypersetup{colorlinks=true, linkcolor=blue, citecolor=red, urlcolor=teal}
\usetikzlibrary{automata} 
\usetikzlibrary{positioning} 
\usetikzlibrary{arrows}

\providecommand{\keywords}[1]{\textbf{\textit{Keywords:}} #1}
\providecommand{\subjclass}[1]{\textbf{\textit{MSC2020:}} #1}
\begin{document}

\nocite{*} 

\title{Decomposition of Spaces of Periodic Functions into Subspaces of Periodic and Antiperiodic Functions and Its Connection to the Rademacher System and the Haar Wavelet Basis}

\author{Hailu Bikila Yadeta \\ email: \href{mailto:haybik@gmail.com}{haybik@gmail.com} }
  \affil{Salale University, College of  Natural Sciences, Department of Mathematics,\\
   Fiche, Oromia, Ethiopia}
\date{\today}
\maketitle

\begin{abstract}
\noindent We prove that the space $\mathbb{P}_p$ of $p$ - periodic functions decomposes as the direct sum $\mathbb{P}_{p/2} \oplus \mathbb{A}\mathbb{P}_{p/2}$, where $\mathbb{P}_{p/2}$ denotes the space of functions periodic with period $p/2$ and $\mathbb{A}\mathbb{P}_{p/2}$ denotes the space of functions antiperiodic with antiperiod $p/2$ (i.e., $f(x+p/2) = -f(x)$). Iterating this decomposition yields a hierarchy of refined periodic subspaces.

Under suitable uniform decay conditions on the residual periodic components, any $p$-periodic function on a compact interval admits a convergent expansion into a series of antiperiodic components with distinct antiperiods. As a concrete example, the continued periodic-antiperiodic decomposition of the fractional part function $\{x\}$ generates the Rademacher system.

Additionally, we examine an orthogonal decomposition of $L^2(0,1)$ induced by reflection symmetry about the midpoint $x = 1/2$, i.e., $f(x) = \pm f(1-x)$. Using explicit projection operators, we show that this reflection-based decomposition generates a multiscale structure analogous to the Haar multiresolution analysis: the antiperiodic (odd-reflection) component yields a system equivalent to the Haar wavelet family $\{\psi_{j,k}\}$, while the periodic (even-reflection) component corresponds to the scaling space of piecewise constant functions. This provides a boundary-condition-based interpretation of the Haar wavelet basis.
\end{abstract}

\noindent\keywords{periodic function, antiperiodic function, direct sum, decomposition, Rademacher system, Haar wavelet, multiresolution analysis }\\
\subjclass{Primary : 42C40, 42A16 ;  Secondary: 39A70, 39A06 }\\

\section{Introduction and preliminaries}
It is evident that any 1-periodic function is also 2-periodic. Consequently, the set of all 1-periodic functions is a subset of the set of all 2-periodic functions. The primary question then becomes: *which functions, other than the 1-periodic ones, are contained within the set of all 2-periodic functions?* This question leads to the study of how spaces of periodic functions can be decomposed into subspaces. The answers to these and similar questions will be addressed in this paper. Before presenting the main results, we provide an example from classical mathematics where a vector space is expressed as a direct sum of its subspaces. Let $\mathcal{F}$ denote the set of all functions
$  f: \mathbb{R}\rightarrow \mathbb{R}  $.
Let us  denote by $\mathbb{E}$ the subspace of all even functions in $\mathcal{F}$ and the subspace of all odd functions in  $\mathcal{F}$ by  $ \mathbb{O} $. Then we have the decomposition
 \begin{equation}\label{eq:oddevendecomposition}
  \mathcal{F} = \mathbb{E} \oplus  \mathbb{O}.
\end{equation}
  According to (\ref{eq:oddevendecomposition}), each  element  $ f \in \mathcal{F} $ can be written as $ f = f_e + f_o  $, where
  \begin{equation}\label{eq:odandevencomponents}
   f_e(x) = \frac{f(x)+f(-x)}{2}  \in \mathbb{E}, \quad  f_o(x) = \frac{f(x)-f(-x)}{2} \in \mathbb{O},
  \end{equation}
The only element common to both $\mathbb{E}$ and $\mathbb{O}$ is the zero function $f(x) = 0$.

The focus  of this paper  is the decomposition of periodic spaces into subspaces of periodic and antiperiodic functions. Let us denote by $ \mathbb{P}_p $, the space of all periodic functions of period $p$ and by $ \mathbb{AP}_p $, the space of all antiperiodic functions of antiperiod $p$ :
  \begin{equation}\label{eq:periodicandantiperiodic}
   \mathbb{P}_p = \{ f\in \mathcal{F}: f(x+p)=f(x) \}, \quad  \mathbb{AP}_p = \{ f\in \mathcal{F}: f(x+p)= -f(x) \}. \end{equation}
The spaces $ \mathbb{P}_p $, and $ \mathbb{AP}_p $  form subspaces of $ \mathcal{F }$.
For $ h \in \mathbb{R} $, we define the shift operator $ E ^h $ and the identity operator $ I $ as:
$$ E ^h y(x) := y(x+h),\quad  I y(x) := y(x) .$$

Let $p >0$ . In terms of the shift operator  $f$ is a $p$-periodic function if  $ E^p f(x)= f(x) $ and  a function $g$ is $p$-antiperiodic if $ E^p g(x)= -g(x) $

In this paper, we demonstrate that any periodic function $f$ with period $p$ can be uniquely decomposed into a sum of a periodic function of period $p/2$ and an antiperiodic function of antiperiod $p/2$.

This work establishes a fundamental connection between elementary periodic-antiperiodic decomposition and two cornerstones of harmonic analysis: the Rademacher system and the Haar wavelet basis

Periodic and antiperiodic functions play an important role in solving linear difference equations. In particular, the general solutions often comprise linear combinations of independent solutions, which are either periodic or antiperiodic functions with some period. See \cite{LB}, \cite{MT}. Additionally, we demonstrate that certain periodic functions can be expressed as an infinite series of antiperiodic functions with varying antiperiods plus a constant.
 \begin{remark}
  Every $p$-antiperiodic function is $2p$-periodic. However not every $2p$-periodic function is a $p$-antiperiodic function. Further properties of $p$-antiperiodic function are available in the literature. Any finite linear combination, or any uniformly convergent infinite series, each of whose terms is a p-periodic (respectively, p-antiperiodic) function, is itself a p-periodic (respectively, p-antiperiodic) function. For example, $ f(x)=\sum_{n =1}^{\infty}\frac{\cos ((2n+1)x ) }{n^2}$ is  $\pi $-antiperiodic function defined by a uniformly convergent  series each of its terms is $\pi $-antiperiodic. See \cite{GN}.
\end{remark}

\begin{remark}
For any fixed $ p > 0$,  the constant function $ f(x)=0 $ is the only function that is both $p$-periodic and $p$-antiperiodic.
\end{remark}

\begin{theorem}
The composition of a periodic or antiperiodic function with an even or odd function satisfies:
\begin{enumerate}
\item[(i)] If $f\in \mathbb{O}$ (odd) and $g\in \mathbb{A}\mathbb{P}_p$ (antiperiodic), then $f\circ g\in \mathbb{A}\mathbb{P}_p$.
\item[(ii)] If $f\in \mathbb{E}$ (even) and $g\in \mathbb{A}\mathbb{P}_p$, then $f\circ g\in \mathbb{P}_p$.
\item[(iii)] If $f\in \mathcal{F}$  and $g\in \mathbb{P}_p$ (periodic), then $f\circ g\in \mathbb{P}_p$.
\end{enumerate}
\end{theorem}

\begin{theorem}
Let $ \omega > 0 $, and $ f \in \mathcal{F} $. Define $g(x)=f(\omega x)$.
\begin{enumerate}
\item[(i)] If $f \in \mathbb{AP}_p$, then $ g \in \mathbb{AP}_{\frac{p}{\omega}}$.
\item[(ii)] If $ f\in \mathbb{P}_p$, then $ g \in \mathbb{P}_{\frac{p}{\omega}}$.
\end{enumerate}
\end{theorem}

\section{Main Results}
\subsection{Decomposition of spaces of periodic functions}

\begin{theorem}\label{eq:Decompositionthm}
  The space $ \mathbb{P}_{p} $ of all $ p $- periodic functions is the direct sum of the space $ \mathbb{P}_{p/2 }$ of all $ p/2 $-periodic function and the space $ \mathbb{AP}_{p/2 } $ of all $ p/2 $-antiperiodic function.
\end{theorem}
\begin{proof}
  Let $ h \in \mathbb{P}_{p} $. Suppose that
   \begin{equation}\label{eq:p1ap1decomposition}
     h(x)= f(x)+ g(x),
   \end{equation}
    for some  $ f \in \mathbb{P}_{ p/2 } $, and  $ g \in \mathbb{AP}_{ p/2}  $. Then
  \begin{equation}\label{eq:shiftedp1ap1decomposition}
     h(x+p/2)= f(x+p/2)+g(x+p/2)= f(x)-g(x).
     \end{equation}
     Then solving (\ref{eq:p1ap1decomposition}) and (\ref{eq:shiftedp1ap1decomposition}) simultaneously we get
     \begin{equation}\label{eq:fandg}
       f(x)= \frac{1}{2}(h(x)+h(x+p/2)), \quad g(x)= \frac{1}{2}(h(x)-h(x+p/2)).
     \end{equation}
     Then $f$ and $g$ defined as in (\ref{eq:fandg}) satisfy the required condition. It remains to show that the representation is unique. Suppose that $ f_1, f _2 \in \mathbb{P}_{p/2} $ and  $ g_1, g_2 \in \mathbb{AP}_{p/2} $ such that $ h = f_1 + g_1= f_2 + g_2 $. Then we have $f_1-f_2 = g_2-g_1$. Since  $f_1-f_2  \in \mathbb{P}_{p/2} $ and $ g_1-g_2  \in \mathbb{AP}_{p/2} $, their equality implies $ f_1-f_2,\, g_1-g_2 \in \mathbb{P}_{p/2} \cap \mathbb{AP}_{p/2}= \{ 0 \}    $. We have  $f_1= f_2$ and $  g_1 = g_2 $.
\end{proof}

We have demonstrated that a periodic function of period $p$ can be decomposed into a periodic function of period $p/2$ and an antiperiodic function of antiperiod $p/2$. This decomposition process can be iterated: starting with the  period $p/2$, we can further decompose it into a periodic function of period $p/4$ and an antiperiodic function of antiperiod $p/4$, and so on.

\begin{definition}
 Let $f$ be a periodic function of period $p$. The\emph{ $n$-th periodic generation of $f$}, denoted by $f_n $,  is the residual periodic function of period $p / 2^n $ obtained after $n$ decompositions of  $f$ and $f_0 = f $. The \emph{ $n$-th antiperiodic generation of $f$}, denoted by $ \tilde{f}_n $,  is an antiperiodic function of antiperiod $p / 2^n $ derived from $f$  after $n$ decompositions.
  \end{definition}

  \begin{remark}
  If $f$ is a $p$-periodic function that is also a $p/2$-antiperiodic function, then $ f \in \mathbb{AP}_{p/2}$. The decomposition in Theorem   \ref{eq:Decompositionthm}, yields $f =\tilde{f_1} + 0 $. That is, the first periodic generation $f_1$ of $f$ is $0$, and the first antiperiodic generation $ \tilde{f_1} $  of $f$ is $f$ itself. Consequently, all subsequent periodic and antiperiodic generations of $f$ are all $0$ s.
\end{remark}

  \begin{theorem}
  Given a periodic function $f$  of period  $p$ the\emph{ $n$-th periodic generation of $f$} is given by
    \begin{equation}\label{eq:nthperiodic}
     f_n = \frac{1}{2^n} \prod_{i=1}^{n} (1+E^{\frac{p}{2^i}}) f.
  \end{equation}
  \end{theorem}

\begin{proof}
We prove by induction on $n$. For $n=0$, the formula gives $f_0 = f$, which is trivially true. Assume the formula holds for some $n\ge 0$. From the decomposition in Theorem 2.1 applied to $f_n$, we have
$$  f_{n+1} = \frac{1}{2}(I + E^{p/2^{n+1}})f_n.  $$
Using the induction hypothesis,
$$ f_{n+1} = \frac{1}{2}\left(I + E^{p/2^{n+1}}\right)\left(\frac{1}{2^n}\prod_{i=1}^{n}(I + E^{p/2^i})f\right)
= \frac{1}{2^{n+1}}\prod_{i=1}^{n+1}(I + E^{p/2^i})f. $$
Thus the formula holds for $n+1$, completing the induction.
\end{proof}

\begin{theorem}
     Given a periodic function $f$ of period $p$, the \emph{$n$-th antiperiodic generation of $f$} is given by
     \begin{equation}\label{eq:nthantiperiodic}
       \tilde{f}_n = \frac{1}{2^n}  ( I-E ^{\frac{p}{2^n}})\prod_{i=1}^{n-1} (1+E^{\frac{p}{2^i}}) f.
     \end{equation}
  \end{theorem}
  \begin{proof}
    The $(n-1)$-th periodic generation  $f_{n-1}$  is decomposed into the $n$-th periodic generation  $f_{n}$ and the $n$-th antiperiodic generation $\tilde {f}_n$. That is $ f_{n-1} = f_n + \tilde{f}_n $. Consequently by (\ref{eq:nthperiodic})
    $$\tilde{f}_n =  f_{n-1}- f_n=  \frac{1}{2^{n-1}} \prod_{i=1}^{n-1} (1+E^{\frac{p}{2^i}}) f -  \frac{1}{2^n} \prod_{i=1}^{n} (1+E^{\frac{p}{2^i}}) f, $$
    which, upon simplification, gives the desired result.
  \end{proof}

  \begin{theorem}
  For each $ n \in \mathbb{N }$,
    $$ f= f_n +  \sum_{k=1}^{n} \tilde{f}_k,$$
    where $f_n $ is the residual periodic part.
  \end{theorem}

\begin{theorem}
Let f be the a $p$-periodic function let $f_n$  denote its $n$th periodic generation( of period $ p/2^n $. If
 $$\lim\limits_{n\rightarrow \infty }\sup \limits_{x_0 \leq x \leq x_0 + p }|f_n(x)|=0, $$
 then $f$ is represented by an infinite series of its antiperiodic generations $ \tilde{f}_n $:
 \begin{equation}\label{eq:infiniteseries}
   f(x) = \sum_{n=1}^{\infty} \tilde{f}_n(x)\quad \text{ for all } x \in [x_0,x_0+p],
 \end{equation}
 and the convergence is uniform on $ [x_0,x_0+p] $.
\end{theorem}

\begin{example}
Consider the 1-periodic function $f(x) = \sin (2 \pi x )$. The decomposition of $f$ yields
  $$ f_1(x)= 0, \quad \tilde{f}_1(x) = f(x) $$
Since $ f \in \mathbb{AP}_{1/2} \subset   \mathbb{P}_{1} $, we have $f(x) = 0 + f(x) $, the desired decomposition. The second generation is the decomposition of $f_1$, and the zero function terminates the process.
\end{example}
Note that the periodic - antiperiodic function decomposition does not always terminate. In the next example and theorem, we show that the decomposition of the fractional part function, which is a $1$-periodic function, continues indefinitely.

\begin{example}
Let us study the periodic decomposition of the fractional part function up to five generations, in a tabular form.
 Antiperiodic component $\tilde{f}_n(x)$ has  antiperiod  $2^{-n}$; Residual periodic component $f_n(x)$ has period $2^{-n}$. After five generations the fractional part function can be  written as:
$$ \{x\} = \frac{\{32x\}}{32} + \frac{31}{64} - \sum_{k=1}^{5} \frac{(-1)^{\lfloor 2^n x \rfloor}}{2^{n+1}}   $$

\begin{table}[H]
\centering
\renewcommand{\arraystretch}{2.0}
\caption{Periodic decomposition of the fractional part function $\{x\}$ (generations 1--5, written with some indicative pattern of generations ). The amplitude of $f_n$ is given as peak-to-peak, while the amplitude of $\tilde{f}_n$ is the ordinary amplitude (half the peak-to-peak).}
\label{tab:fractional_decomp}
\begin{tabular}{cccccc}  
\toprule
$n$ & $f_n(x)$ (periodic part) & $\tilde{f}_n(x)$ (antiperiodic part) & Period & $f_n$ amplitude  & $\tilde{f}_n$ amplitude \\
\midrule
1 & $\displaystyle \frac{\{2x\}}{2} + \frac{2-1}{2^2}$   & $\displaystyle -\frac{(-1)^{\lfloor 2x\rfloor}}{2^2}$   & $2^{-1}$ & $\frac{1}{2}$ & $\frac{1}{2^2}$ \\[3pt]
2 & $\displaystyle \frac{\{2^2x\}}{2^2} + \frac{2^2-1}{2^3}$   & $\displaystyle -\frac{(-1)^{\lfloor 2^2x\rfloor}}{2^3}$   & $2^{-2}$ & $\frac{1}{4}$ & $\frac{1}{2^3}$ \\[3pt]
3 & $\displaystyle \frac{\{2^3x\}}{2^3} + \frac{2^3-1}{2^4}$  & $\displaystyle -\frac{(-1)^{\lfloor 2^3x\rfloor}}{2^4}$  & $2^{-3}$ & $\frac{1}{2^3}$ & $\frac{1}{2^4}$ \\[3pt]
4 & $\displaystyle \frac{\{2^4x\}}{2^4} + \frac{2^4-1}{2^5}$ & $\displaystyle -\frac{(-1)^{\lfloor 2^4x\rfloor}}{2^5}$ & $2^{-4}$ & $\frac{1}{2^4}$ & $\frac{1}{2^5}$ \\[3pt]
5 & $\displaystyle \frac{\{2^5 x\}}{2^5} + \frac{2^5-1}{2^6}$ & $\displaystyle -\frac{(-1)^{\lfloor 2^5x\rfloor}}{2^6}$ & $2^{-5}$ & $\frac{1}{2^5}$ & $\frac{1}{2^6}$ \\
\bottomrule
\end{tabular}
\end{table}
\end{example}


\begin{theorem}
  The $n$-th periodic generation of the $1$-periodic fractional part function $f(x)= \{x\}$ is the $\frac{1}{2^n}$-periodic function given by:
  \begin{equation}\label{eq:nthperforfraction}
   f_n(x)= \frac{1}{2^n}  \sum_{k=0}^{2^n-1}\left\{x+ \frac{k}{2^n}\right \}.
\end{equation}
\end{theorem}

\begin{proof}
   We prove by induction over $n$. For $n=0$ the result yields $f_0(x)=f(x)$, which is the given function itself. For $n=1$ we get
   $$f_1(x)= \frac{\{x\}+\{x+1/2\}}{2} = \frac{1}{2}\sum_{k=0}^{1}\left\{x+ \frac{k}{2}\right \}, $$
   which is the desired result according to (\ref{eq:fandg}). Now suppose that, for arbitrary $n\in \mathbb{\mathbb{N}} $, the equation (\ref{eq:nthperforfraction}) holds true. According to (\ref{eq:fandg}),

   \begin{align}\label{eq:nplus1step}
     f_{n+1}(x) & =  \frac{1}{2} \left( f_n(x) + f_n(x +1/2^n)   \right)  \nonumber \\
      & = \frac{1}{2} \left( \frac{1}{2^{n}}  \sum_{k=0}^{2^n-1}\left\{x+ \frac{k}{2^n}\right \}  + \frac{1}{2^{n}}  \sum_{k=0}^{2^n-1} \left\{x+ \frac{k}{2^{n}} + \frac{1}{2^{n+1}} \right \}\right).
   \end{align}
Simplifying (\ref{eq:nplus1step}), which is a refined sum of (\ref{eq:nthperforfraction}), yields

  \begin{equation}\label{eq:nplusgenoffraction}
    f_{n+1}(x)= \frac{1}{2^{n+1}} \sum_{k=0}^{2^{n+1}-1} \left\{x+ \frac{k}{2^{n+1}}\right \}.
  \end{equation}
  This proves the general formula (\ref{eq:nthperforfraction}) also applies for $n+1$. Hence the formula (\ref{eq:nthperforfraction}) is proved to be true for all $n \in \mathbb{N }$.
\end{proof}

\begin{theorem}
  The $n$-th antiperiodic generation of the $1$-periodic function  $f(x)=\{x\}$ is  the $\frac{1}{2^n}$-antiperiodic function given by
  \begin{equation}\label{eq:nthantiperforfraction}
  \tilde{f}_n(x)=   \frac{1}{2^n}  \sum_{k=0}^{2^n-1} (-1)^k \left\{x+ \frac{k}{2^n}\right \}.
  \end{equation}
\end{theorem}

\begin{proof}
By (\ref{eq:fandg}), (\ref{eq:nthperforfraction}), we get
  \begin{align*}
    \tilde{f}_n(x) =  \frac{f_{n-1}(x)-f_{n-1}\left(x+ \frac{1}{2^n} \right)}{2} &
      =\frac{1}{2^n}  \sum_{k=0}^{2^{n-1}-1} \left(   \left\{ x+\frac{k}{2^{n-1}}  \right\}  -\left\{ x+\frac{k}{2^{n-1}}  +\frac{1}{2^n}   \right\} \right)     \\
     & = \frac{1}{2^n} \sum_{k=0}^{2^n -1}(-1)^k \left\{ x+\frac{k}{2^{n}}  \right\}.
  \end{align*}

\end{proof}

\begin{theorem}
   Let $f_n(x)$ is the $n$-th periodic generation of the $1$-periodic function $f(x)= \{x\}$ given by (\ref{eq:nthperforfraction}). Then
  $$ \lim_{n \rightarrow \infty } f_n(x)= \frac{1}{2}$$
\end{theorem}

\begin{proof}
   We proceed to evaluate the required limit as the limit of a  Riemann sum of  a Rieman-integrable function.
\begin{align*}
    \lim_{n \rightarrow \infty } f_n(x) & =  \lim_{n \rightarrow \infty }\frac{1}{2^n}  \sum_{k=0}^{2^n-1}  \left\{x+ \frac{k}{2^n}\right \} \\
     & = \lim_{n \rightarrow \infty } \left(\frac{1}{2^n}  \sum_{k=0}^{2^n}  \left\{x+ \frac{k}{2^n}\right \} - \frac{\{x\}}{2^n}\right)\\
     & = \lim_{n \rightarrow \infty } \frac{1}{2^n}  \sum_{k=0}^{2^n}  \left\{x+ \frac{k}{2^n}\right \} = \int_{x}^{x+1}\{s\}ds= \int_{0}^{1}\{s\}ds= \int_{0}^{1}s\, ds= \frac{1}{2}
  \end{align*}
  \end{proof}

\begin{theorem}
   The $n$-th antiperiodic generation of the fractional part function $\{x\}$, given in (\ref{eq:nthantiperforfraction}), can be written alternatively, without summation, as

  \begin{equation}\label{eq:nthantipergenenwithoutsum}
    \tilde{f}_n(x) = - \frac{(-1)^ \lfloor 2^n x \rfloor }{2^{n+1}}.
  \end{equation}
\end{theorem}

\begin{proof}

We now prove that
\begin{equation}\label{eq:equalityoftwoforms}
   \frac{1}{2^n}  \sum_{k=0}^{2^n-1} (-1)^k \left\{x+ \frac{k}{2^n}\right \} =   - \frac{(-1)^ \lfloor 2^n x \rfloor }{2^{n+1}} .
\end{equation}

It suffices to show that they are equal on the interval $[0,\frac{1}{2^{n-1}} ) $  because both sides are $2^{1-n}$-
periodic (the left-hand side by the periodicity argument given above, the right hand side of \eqref{eq:nthantipergenenwithoutsum}  by adding $2^{1-n}$ changes by $\lfloor 2^n x \rfloor $ by $2$  and that $(-1)^2= 1  $ ). We consider two subintervals.

\textbf{Case 1}: Let  $0 \leq x \frac{1}{2^n}$. Then $0 \leq  2^n x < 1 $. Therefore, $ \lfloor 2^n x    \rfloor  = 0 $, and the right hand side is $ -\frac{1}{2^{n+1}}$.

For $0 \leq x <  \frac{1}{2^n}$ we have
$$  0 \leq x+ \frac{k}{2^n} < \frac{k+1}{2^n}  \leq 1 , \quad k=0,1,\cdots ,2^n-1.       $$

Therefore, $\left\{x+ \frac{k}{2^n} \right \} =  x + \frac{k}{2^n} $. Consequently

\begin{align*}
  \frac{1}{2^n}  \sum_{k=0}^{2^n-1} (-1)^k \left\{x + \frac{k}{2^n} \right \} &= \frac{1}{2^n}  \sum_{k=0}^{2^n-1}(-1)^k x   + \frac{1}{2^{2n} }\sum_{k=0}^{2^n-1} (-1)^k k \\
   & = \left(\frac{1}{2^n} \times  0\right)  +  \left(\frac{1}{2^{2n}}  \times   - 2^{n-1}\right)  = -\frac{1}{2^{n+1}}.
  \end{align*}
Note that the first summation  $ \sum_{k=0}^{2^n-1}(-1)^k x  $ is equal to $0$ because it  has $ 2^n$ terms   with alternating  consecutive terms that cancel out. For the second summation $\sum_{k=0}^{2^n-1} (-1)^k k $, we have
\begin{align*}
 \sum_{k=0}^{2^n-1} (-1)^k k &= (0-1) + (2-3)+ \cdots + (2^n-2)-(2^n-1)  \\
   & = -1 -1 \cdots -1 \\
   &  =  -2^{n-1}.
\end{align*}
Since the original series has $2^n$ terms grouped into $2^{n-1}$ pairs.

\textbf{Case 2}: Let  $    \frac{1}{2^n} \leq x   <  \frac{1}{2^{n-1}}$. Then $ 0 < \frac{k+1}{2^n} \leq x + \frac{k}{2^n} < \frac{k+2}{2^n} < 1, \quad k=0,1,\cdots ,2^n-2 $.  However, for the last index $k = 2^n -1 $  we have $ 1   \leq   x + \frac{k}{2^n} < 1 + \frac{1}{2^n } $. Consequently
$$   \left\{x+ \frac{k}{2^n} \right \} = \begin{cases}
         &  x + \frac{k}{2^n}, \mbox{  if  }  0 \leq k \leq 2^n-2,             \\
         &       x + \frac{1}{2^n}-1,     \mbox{ for } k=2^n-1 .
      \end{cases}
$$
Therefore, for  $ \frac{1}{2^n} \leq x   <  \frac{1}{2^{n-1}}$
\begin{align*}
    \frac{1}{2^n}  \sum_{k=0}^{2^n-1} (-1)^k \left\{x+  \frac{k}{2^{n}}\right \} & = \frac{1}{2^n}  \sum_{k=0}^{2^n-2} (-1)^k \left (x+  \frac{k}{2^{n}}\right ) - \frac{1}{2^n}\left(x-\frac{1}{2^n}\right) \\
   & = \frac{1}{2^n}\left(  x\sum_{k=0}^{2^n-2} (-1)^k +  \frac{1}{2^n}\sum_{k=0}^{2^n-2}(-1)^k k      \right) +\frac{1}{2^{2n}}-\frac{x}{2^n}\\
   & = \frac{1}{2^n}\left(  x  +  \frac{1}{2^n}  \left( -2^{n-1} +2^n-1     \right)    \right) +\frac{1}{2^{2n}}-\frac{x}{2^n}\\
   & = \frac{x}{2^n }+ \frac{1}{2^{n+ 1}} - \frac{1}{2^{2n}}  + \frac{1}{2^{2n}}-\frac{x}{2^n}\\
   & = \frac{1}{2^{n+1}}.
\end{align*}
Now let us evaluate the expression on the right hand side of \eqref{eq:equalityoftwoforms}. For $ \frac{1}{2^n} \leq x   <  \frac{1}{2^{n-1}} $, we have $ 1 \leq 2^nx < 2 $. Therefore,  $\lfloor  2^n x \rfloor = 1 $, and consequently,
 $$  -\frac{(-1)^ \lfloor 2^n x \rfloor }{2^{n+1}}  =  \frac{1}{2 ^{n+1}}. $$
Finally for $ 2^{1-n}$ periodicity of  the  expression on the left hand side of \eqref{eq:equalityoftwoforms}
\begin{align*}
    \frac{1}{2^n}  \sum_{k=0}^{2^n-1} (-1)^k \left\{x+ \frac{k}{2^n}+ \frac{1}{2^{n-1}}\right \} & = \frac{1}{2^n}  \sum_{k=0}^{2^n-1} (-1)^k \left\{x+  \frac{k+2}{2^{n}}\right \}\\
     & =   \frac{1}{2^n}  \sum_{k=2}^{2^n+1} (-1)^k \left\{x+  \frac{k}{2^{n}}\right \} \\
     & =  \frac{1}{2^n}  \sum_{k=0}^{2^n-1} (-1)^k \left\{x+  \frac{k}{2^{n}}\right \}.
  \end{align*}

Note that the terms corresponding to the index $k=0$ and $k=2^n$ are equal, as are those for $k=1$ and $k=2^n+1 $. This proves the assumed periodicity. Now we show that
$$ \frac{1}{2^n}  \sum_{k=1}^{2^n-1} (-1)^k \left\{x+  \frac{k}{2^{n}}\right \}  = \frac{(-1)^ \lfloor 2^n x \rfloor }{2^{n+1}} = \begin{cases}
                & \frac{-1}{2^{n+1}},   \mbox{ if }  0 \leq x < \frac{1}{2^n}, \\
                &  \frac{-1}{2^{n+1}},   \mbox{ if }  \frac{1}{2^n} \leq x < \frac{1}{2^{n-1}}.
             \end{cases}
$$

\end{proof}

\begin{example}
  Based on the  propositions from the previous examples, the fractional part function  $\{x\}$  admits  the infinite series representation  involving  periodic and antiperiodic components:
  \begin{equation}\label{eq:infiniteseriesforfraction}
    \{x\} = \frac{1}{2}-\sum_{n=1}^{\infty} \frac{(-1)^ \lfloor 2^n x \rfloor }{2^{n+1}}
  \end{equation}
 In addition, the Fourier series representation  of $ \{x\} $ is given by:
  \begin{equation}\label{eq:Fourierseries}
    \{x\} = \frac{1}{2}-\frac{1}{\pi}\sum_{n=1}^{\infty}\frac{\sin( 2 \pi n x )}{n}, \quad x \notin \mathbb{Z} .
  \end{equation}
  Furthermore, the even-odd decomposition of $ \{x\} $ is expressed as:
  \begin{equation}\label{eq:evenoddforfractional}
  \{x\} = \underbrace{\frac{1}{2}}_{\text{even}} + \underbrace{ \left(\{x\} - \frac{1}{2}\right)}_{\text{odd}}
  \end{equation}
\end{example}

\subsection{The Orthogonality conditions of sequences of antiperiodic generations}

Let $L^2(0,1) $ be the  Hilbert space of all square integrable complex-valued functions on the open interval $(0,1)$ with  the standard  $L^2[0, 1] $ inner product:
\begin{equation}\label{eq:IPinL2}
  \langle f, g \rangle = \int_0^1 f(x) \overline{g(x)}  dx.
\end{equation}
In this subsection, we show that set
\begin{equation}\label{eq:specificbasis}
   S = \{ (-1)^{\lfloor 2^n x \rfloor}, n \in \mathbb{N} \}
\end{equation}
of all antiperiodic families of functions generated from the fractional part function  $\{x\}$  is  an orthogonal set under the inner product \eqref{eq:IPinL2} and $ S $ is the Radematcher system.

\begin{theorem}
Let  $ f_n \in L^2[0,1],\,  n \in \mathbb{Z}_{\geq 0}$   is defined by   $ f_n(x)=  (-1)^{\lfloor 2^n x \rfloor}$. Then,
\begin{equation}\label{eq:innerproduct}
   \langle f_n, f_m \rangle =  \int_0^1 (-1)^{\lfloor 2^m x \rfloor + \lfloor 2^n x \rfloor}  dx = \delta_{m,n},
\end{equation}
  where $ \delta_{m,n} $ is the Kroneker delta function given by  $ \delta_{m,n} = 1 $ if $m =n $ and $ 0 $ otherwise.
 \end{theorem}

\begin{proof}
If $m =n $, then
$$  \langle f_n ,f_n \rangle  = \int_{0}^{1}(-1)^{\lfloor 2^m x \rfloor + \lfloor 2^m x \rfloor} dx =  \int_{0}^{1}(-1)^{2\lfloor 2^m x \rfloor} dx = \int_{0}^{1} 1 dx =1   $$
Let   $ m \neq n $.  Assume $   m < n $. Partition  the interval $[0, 1)$  into $2^n $ subintervals:
$$    \left[ \frac{k}{2^n}, \frac{k+1}{2^n} \right), \quad k = 0,\, 1, \dots, 2^n - 1.    $$
On each interval, $\lfloor 2^n x \rfloor = k $. Write $k = 2^{n-m} j + r $, where

$$ j = \lfloor k / 2^{n-m} \rfloor, \quad  r \in \{0, \dots, 2^{n-m}-1\} . $$
 Then \(\lfloor 2^m x \rfloor = j\), and  $    (-1)^{\lfloor 2^m x \rfloor + \lfloor 2^n x \rfloor} = (-1)^{j + k}.  $. The integral becomes:
$$   \langle f_m, f_n \rangle = \frac{1}{2^n} \sum_{k=0}^{2^n - 1} (-1)^{j + k}.    $$
Grouping by $j$ (with $2^{n-m}$ values per $j$):
$$
\sum_{r=0}^{2^{n-m} - 1} (-1)^{j + (2^{n-m} j + r)} = (-1)^j (-1)^{2^{n-m} j} \sum_{r=0}^{2^{n-m} - 1} (-1)^r = (-1)^j (1) \cdot 0 = 0,
$$
since the sum over $r$ has an even number of alternating $\pm 1$. This completes the proof.
\end{proof}

\begin{remark}
  We have demonstrated that the fractional part function $\{x\}$ exhibits an infinite series representation, where terms are  periodic and antiperiodic functions. However, it is not always the case that the set $S$, together with the set $\{1\}$,  forms a basis for  arbitrary $1$- periodic function. For example, consider the  function $f(x)= \sin (2 \pi x ) $; this serves as a counterexample. The set $S$ is not complete in generating $L^2[0,1]$. In fact, it can be shown that the set $S$ is a Rademacher system ( see \cite{AS}).
   \end{remark}

   \begin{definition}
   A   \textbf{ Rademacher system }   is the orthonormal system defined on $[0,1]$ as
  $$ r_k (x) :=\sgn \sin 2^{k} \pi x,\quad  x  \in [0,1], n \in \mathbb{N} .   $$
  See, for example, \cite{AS}.
   \end{definition}

   \begin{theorem}
   Except for the set of dyadic rationals, for all $ x  \in [0,1]$
   $$  (-1)^{\lfloor 2^k x \rfloor}  = \sgn \sin 2^k \pi x : =  r_{k} (x). $$
   \end{theorem}

\subsection{The Periodic-Antiperiodic Decomposition of the Hilbert Space $L^2[0,1]$}

In this subsection we illustrate that the Hilbert space $L^2[0,1]$,  with inner product \eqref{eq:IPinL2}, can be  written as a direct sum of the reflection-even subspace $\mathbb{P}$ and reflection-odd subspace $ \mathbb{A} $.
\begin{equation}\label{eq:PsubsetofL2}
   \mathbb{P} = \{ f \in L^2 (0,1): f(x)=f(1-x),\quad  a.e.  \text{ in } [0,1]   \}
\end{equation}

\begin{equation}\label{eq:AsubsetofL2}
 \mathbb{ A} = \{ f \in L^2 (0,1): f(x) =  -f(1-x),\quad  a.e.  \text{ in } [0,1]  \}
\end{equation}

\begin{remark}
To avoid confusion with the translation-periodic and translation-antiperiodic notions of Section 2.1, we refer to the subspaces $ \mathbb{P} $ and $ \mathbb{A }$ as the reflection-even and reflection-odd subspaces, respectively.
For continuous representatives these imply the boundary conditions $ f(0) =f(1) $ (periodic) and $ f(0)=-f(1) $ (antiperiodic) on $[0,1]$, but should not be confused with the translation-based decomposition of Section 2.1.
\end{remark}

\begin{theorem}
We have the following direct sum decomposition:
\begin{equation}\label{eq:L2decompostion}
   L^2[0,1] =  \mathbb{P}  \oplus \mathbb{A}
\end{equation}
\end{theorem}

\begin{proof}
 Let  $f \in L^2 [0, 1] $. Then
  \begin{equation}\label{eq:decompostioninH1}
    f(x)= f_P(x)+ f_A(x),
  \end{equation}
  where
\begin{equation}\label{eq:fAandfP}
  f_P(x)= \frac{f(x)+f(1-x)}{2} \in \mathbb{P} , \quad  f_A(x)= \frac{f(x)- f(1-x)}{2}\in \mathbb{A}.
\end{equation}
If $ f \in  \mathbb{P}  \cap  \mathbb{A} $, then $f(x)= f(1-x)= - f(1-x)$. Hence $ f(x)= 0,  \text{a.e.  in } [0,1] $. So,  $ \mathbb{P}  \cap \mathbb{A}= \{ 0 \} $.
\end{proof}

\begin{theorem}
  Let $ I, S :   L^2[0, 1] \to L^2[0, 1]  $, where  $(Sf)(x)= f(1-x),  $ and $I$ is the identity. Then
  \begin{itemize}
    \item $S$ is an involution, that is $S^2 = I $,
    \item  $ \frac{1}{2} (I \pm S) $ are projections, that is $ \left [\frac{1}{2} (I \pm S) \right ] ^2 =\frac{1}{2} (I \pm S) $.
  \end{itemize}
\end{theorem}

\begin{remark}
  For the Hilbert space $L^2[a, b]$, the operator $S$ is defined as $(Sf)(x)= f(a + b-x) $.
\end{remark}

\begin{theorem}
  The subspaces $A$ and $P $ of the Hilbert space $L^2[0,1]$  defined in \eqref{eq:PsubsetofL2}  and \eqref{eq:AsubsetofL2} are closed subspaces.
\end{theorem}
\begin{proof}
We show that $S: L^2[0,1] \to  L^2[0,1] $ is a bounded linear operator. Indeed,
  $$ \|Sf\|^2 = \int_{0}^{1} |f(1-x)|^2 dx = \int_{0}^{1} |f(x)|^2 dx= \|f\|^2,  $$
  so that $S$ is an isometry. The subspace $ \mathbb{A} $ is the kernel of $I+S $ in $L^2[0,1] $, and the subspace $ \mathbb{P} $ is the kernel of $ I+ S $ in $L^2[0,1] $. Both the operators $I + S $ and $I - S $ are linear and bounded. We know that the kernel of a continuous linear operator between Banach spaces is closed. See, for example \cite{BPT}.
\end{proof}

\begin{theorem}
  For any $f, g \in L^2[0,1] $. Then $ \langle f_P, g_A \rangle = \langle f_A, g_P \rangle = 0  $
\end{theorem}

\begin{proof}
By \eqref{eq:PsubsetofL2} and \eqref{eq:AsubsetofL2}, using properties of definite integrals, we get
  $$ \langle f_P, g_A \rangle = \int_{0}^{1} f_P(x)g_A(x)dx   = \int_{0}^{1} f_P(1-x)g_A(1-x)dx = - \int_{0}^{1} f_P(x)g_A(x)dx = - \langle f_P, g_A \rangle .   $$
   So $ \langle f_P, g_A \rangle = 0 $. In similar way, or by arbitrariness of $f$ and $g$, $ \langle f_A, g_P \rangle = 0 $.
\end{proof}

\begin{corollary}
  For any $f, g \in L^2[0,1] $, if $f= f_A + f_P, \, g = g_A + g_P$, then $ \langle f, g \rangle =  \langle f_P, g_P \rangle +  \langle f_A, g_A \rangle $.
\end{corollary}

\begin{corollary}
  For any $  g \in L^2[0,1] $ with  the decomposition $f = f_P(x)+ f_A(x)$,
  $$ \|f\|^2=  \|  f_P  \|^2 + \|  f_A  \|^2   $$
\end{corollary}

\begin{theorem}
  If $ f \in \mathbb{P} \cap C[0,1] $, then the property $ f(x) =f(1-x) $ holds true for every $ x\in [0,1] $.  If $ f \in \mathbb{A }\cap C[0,1] $, then  the property $ f(x) = -f(1-x) $ holds true for every $ x\in [0,1] $.
\end{theorem}

\begin{corollary}
   If $ f \in  C[0,1] $, then we have the boundary periodicity and boundary antiperiodicity conditions given by :
   \begin{equation}\label{eq:boundaryconditions}
   f_P(0)= \frac{f(0)+f(1)}{2} = f_P(1), \quad f_A(0)= \frac{f(0)-f(1)}{2} = - \frac{f(1)-f(0)}{2} = -f_A(1).
\end{equation}
\end{corollary}
Let us examine the connection between different Fourier series representations of a function $f \in L^2[0,1]$. We know that the set
\begin{equation}\label{eq:basisofL2}
 \mathcal{B} = \{1, \cos (n  \pi   x ), \sin ( n  \pi   x ), n \in \mathbb{N}  \}
 \end{equation}
is an orthogonal basis of the Hilbert space $L^2[0,1]$. Consequently, the Fourier series representation of $f \in L^2[0,1]$ takes the form:
\begin{equation}\label{eq:FouriersseriesinL^2}
     f(x) = a_0 +  \sum_{n=1}^{\infty}    a_n  \cos (n  \pi   x )+ b_n  \sin ( n  \pi   x ),
 \end{equation}

\begin{remark}
The expansion  is sometimes called the \textit{full Fourier series} of $f$ on $[0,1]$. It differs from other expansions that arise when we extend $f$ to the symmetric interval $(-1,1)$ in particular ways:
\begin{itemize}
\item If we extend $f$ to an \textbf{even} function on $(-1,1)$ (i.e., $f(-x)=f(x)$ for $x\in(0,1)$), we obtain the Fourier cosine series (all $b_n = 0$).
\item If we extend $f$ to an \textbf{odd} function on $(-1,1)$ (i.e., $f(-x)=-f(x)$), we obtain the Fourier sine series (all $a_n = 0$).
\item If we extend $f$ \textbf{periodically with period 1} to $(-1,1)$ (i.e., $f(x+1)=f(x)$), the resulting Fourier series on $(-1,1)$ contains only even harmonics ($n$ even), and the coefficients simplify accordingly (the odd-index coefficients vanish).
\end{itemize}
Thus, the full Fourier series (with all harmonics present) does not correspond to any special symmetry; it is simply the orthogonal expansion of $f$ with respect to the basis $\mathcal{B}$ on $[0,1]$.
\end{remark}

 \begin{theorem}
Let $f\in L^2[0,1]$ is represented by a Fourier series of the form given in \eqref{eq:FouriersseriesinL^2}.
 Then  Fourier series for the components $f_P$ and $f_A$ corresponding to the decomposition $f(x)= f_P(x) + f_A(x) $ given in  \eqref{eq:L2decompostion}, are given by:
 \begin{equation}\label{eq:periodicboundaryexpansion}
 f_P(x) = a_0 +  \sum_{n=1}^{\infty}    a_{2n}  \cos (2n  \pi   x )+ b_{2n-1}  \sin ( (2n-1)  \pi   x ) \in \mathbb{P},      \end{equation}

\begin{equation}\label{eq:antiperiodicboundaryexpansion}
  f_A(x) =   \sum_{n=1}^{\infty}    a_{2n-1}  \cos ((2n-1)  \pi   x )+ b_{2n}  \sin ( (2n)  \pi   x ) \in \mathbb{A}.
  \end{equation}
  \end{theorem}

 \begin{proof}
   We observe that the function $ f_P$  and $f_A$ satisfy the conditions given in \eqref{eq:PsubsetofL2} and \eqref{eq:AsubsetofL2}.

  $$ S (\cos (n \pi x ))=  \cos ( n \pi (1-x))= \cos (n \pi x - n \pi )= (-1)^n \cos (n \pi x) $$

   $$ S (\sin (n \pi x ))=  \sin ( n \pi (1-x))= \sin (n \pi x - n \pi )= -(-1)^n \sin (n \pi x) $$

   $$ f_P(x)= \frac{1}{2}(f(x)+(Sf)(x))=  a_0 + \sum_{n=1}^{\infty} \frac{1}{2}(a_n(1+(-1)^n)\cos(n \pi x) + b_n(1-(-1)^n)\sin(n \pi x) ), $$

   $$ f_A(x)= \frac{1}{2}(f(x)-(Sf)(x))=   \sum_{n=1}^{\infty} \frac{1}{2}(a_n(1-(-1)^n)\cos(n \pi x) + b_n(1+ (-1)^n)\sin(n \pi x) ), $$
   The desired results follows by the fact $ 1-(-1)^n $ is equal to $2$ for odd $n$ and $0$ for even $n$; and $ 1+(-1)^n $ is equal to $0$ for odd $n$ and $2$ for even $n$.
   \end{proof}

 \begin{corollary}
   If $f$ is continuous on $[0,1]$, and its Fourier series converges at $x = 0, 1 $, we get periodic  and antiperiodic boundary conditions:
  \begin{equation}\label{eq:summativeBC}
    f_P(0) =  \sum_{n=0}^{\infty} a_{2n}  = f_P(1), \quad  f_A(0) =  \sum_{n=1}^{\infty} a_{2n-1}  = -f_A(1)
   \end{equation}
 \end{corollary}
 Once again, from the Fourier series representation of $f \in L^2[0,1]$,  due to the condition of the  pairwise orthogonality of the distinct pairs of elements of the basis \eqref{eq:basisofL2} of $L^2[0,1]$, it follows  that for any $g \in \mathbb{P }$ and for any $h \in \mathbb{A} $ :
   \begin{equation}\label{eq:orohgonalityfromcomplmentspace}
      \langle g, h \rangle = 0
   \end{equation}
  In particular, for the periodic antiperiodic decomposition $f \in L^2[0,1]$ as $f = f_P + f_A $
  \begin{equation}\label{eq:orthogonalityofcomponents}
  \langle f_P, f_A \rangle = 0
  \end{equation}

\subsection{Connection of the periodic-antiperiodic boundary decomposition to Haar Wavelets}

The starting function or the Haar scaling function is defined as:
\begin{equation}\label{eq:scalingfunction}
 \phi(x)= \begin{cases}
             &  1 \mbox{   if } 0 \leq x < 1,  \\
             &  0  \mbox{ otherwise }.
          \end{cases}
\end{equation}
The Haar wavelet (mother wavelet) is given by
\begin{equation}\label{eq:thepsiPhirelation}
  \psi(x)=\phi(2x)-\phi(2x-1)= \begin{cases}
                                  &  1    \mbox{   if } 0 \leq x < \frac{1}{2}, \\
                                 &  -1    \mbox{   if } \frac{1}{2} \leq x < 1, \\
                                 &   0    \text{    otherwise },
                               \end{cases}
\end{equation}
see, for example \cite{STDD}. For a fixed  $j \geq 0$, the function $\psi(2^j x)$ is supported in the interval $0 \leq x < \frac{1}{2^j}$. It is for this case that as $0 \leq 2^j x < 1$ and that $\psi(x)=0$ outside of $[0,1]$, as defined in \eqref{eq:thepsiPhirelation}. By shifting $\psi_{j,0} = \psi(2^j x)$ by $k/2^j$ units to the right, we get $\psi_{j,k}(x)= \psi(2^j x - k)$ for $k$ values ranging from $0, 1, \dots, 2^j - 1$. The supports of all these functions are disjoint and all  fit into the unit interval $[0,1]$. The formula for Haar wavelet basis functions is given by
\begin{equation}\label{eq:Haarwavletbasis}
  \psi_{j,k}=2^{\frac{j}{2}}\psi(2^jx-k),\quad  j\geq 0, \quad  k =0, 1, \cdots ,2^j-1,
\end{equation}
comes from two natural operations : dilation(scaling) and translations(shifting). See, for example, \cite{CA}, \cite{MS}, \cite{ID}. Combining with the normalization to keep the $L^2$-norms equal to 1. The second level wavelets are :

\begin{equation}\label{eq:psi10andpsi11}
   \psi_{1,0}(x)=  \sqrt{2} \cdot  \begin{cases}
                                  &  1    \mbox{   if } 0 \leq x < \frac{1}{4}, \\
                                   &  -1    \mbox{   if } \frac{1}{4} \leq x < \frac{1}{2},  \\
                                  &  0     \mbox{ otherwise }.
                               \end{cases},  \quad  \psi_{1,1}(x)=  \sqrt{2} \cdot  \begin{cases}
                                  &  1    \mbox{   if } \frac{1}{2} \leq x < \frac{3}{4}, \\
                                   &  -1    \mbox{   if } \frac{3}{4} \leq x < 1,  \\
                                  &  0     \mbox{ otherwise }.
                               \end{cases}
\end{equation}

We want  each $\psi_{j,k}$ to have  unit norm. The function $ \psi(2^j x-k) $ has support of length  $2^{-j} $. Since its values are $\pm 1 $, $  \|\psi(2^j x-k) \|_2^2 = 2^{-j} $. Therefore
\begin{equation}\label{eq:normofpsysubj}
  \|\psi(2^j x-k) \|= 2^{-j/2}
\end{equation}
The  normalised function is obtained by multiplying by $ 2^{j/2}$, so that
\begin{equation}\label{eq:normalisedjsubk}
 \|   2^{j/2}\psi(2^j x-k) \| = 1.
\end{equation}
So the normalized dilated wavelet is  $  2^{j/2}\psi(2^j x-k) $.
To cover the whole unit interval $[0,1)$ we need to translate (shift) the dilated wavelet to different positions.
$ \psi_{j,k}(x)= 2^{j/2}\psi_{j,k}(2^j x -k)$ is non-zero only when $0 \leq 2^j x -k < 1  $ or $k/ 2^j \leq x < (k+1)/2^j$.
Now we associate this Haar wavelets with the sequence of anti periodic decomposition of  a unit rectangle function $\chi_{[0,1)}$. Consider the piecewise constant function:

\begin{equation}\label{eq:aandb}
 f(x)= \begin{cases}
             &  a \mbox{   if }, 0 \leq x < 1/2,  \\
             &  b  \mbox{ if }, 1/2 \leq x < 1.
          \end{cases}
\end{equation}
Then

\begin{equation}\label{eq:psiofandb}
   f_A(x)= \frac{f(x)- f(1-x)}{2}= \frac{a-b}{2} \cdot  \begin{cases}
                                  &  1    \mbox{   if } 0 \leq x < \frac{1}{2}, \\
                                   &  -1    \mbox{   if } \frac{1}{2} \leq x < 1.
                                  \end{cases}
\end{equation}

Therefore $ f_A(x) = \frac{a-b}{2} \psi (x) $, yields the mother wavelet up to a scalar multiple, where as $f_p(x)= \frac{a+b}{2} \chi_{[0,1)}(x) $ which is the scaling function of the Haar system up to scalar multiple. Now consider the method of periodic antiperiodic boundary condition decomposition  That yields the Haar Wavlet basis.
Let  the scaling function be the same as  the scaling function of  the Haar wavelet  given in \eqref{eq:Haarwavletbasis}. Divide the interval $[0,1)$ into two subintervals as $ [0, 1/2) $ and $[1/2,1) $.  Next divide each of these two half subintervals each of them into two parts as $ [0,1/4),  [1/4,1/2),  [1/2,3/4), [3/4,1)$. Continuing in this manner, at the $j$  step, we have $2^j $ subintervals of $[0,1)$ whose boundary points are dyadic rationals. Define

\begin{equation}\label{eq:piecewiseconstamt}
    I_{j,k}:= \left [\frac{k}{2^{j-1}},\frac{k+1}{2^{j-1}}\right ), \quad k = 0, 1,\cdots, 2^{j-1}-1.
\end{equation}

For each $j \in \mathbb{N} $ we have $2^{j-1}$ antiperiodic cycles.

\begin{equation}\label{eq:fsubjkdefined}
f_{j,k}(x) =
\begin{cases}
 1, & \text{if } x \in \left[ \dfrac{k}{2^{j-1}}, \dfrac{2k+1}{2^{j}} \right), \\[8pt]
-1, & \text{if } x \in \left[ \dfrac{2k+1}{2^{j}}, \dfrac{k+1}{2^{j-1}} \right), \qquad k = 0,1,\dots, 2^{j-1}-1.
\end{cases}
\end{equation}

\begin{theorem}
If $x \in \left[ \dfrac{k}{2^{j-1}}, \dfrac{2k+1}{2^{j}} \right)$ then $\dfrac{2k+1}{2^{j-1}} - x \in \left( \dfrac{2k+1}{2^{j}}, \dfrac{k+1}{2^{j-1}} \right]$, and conversely.
\end{theorem}

\begin{theorem}
The following  reflection condition:
\begin{equation}\label{eq:reflectioncondition}
   f_{j,k}\left( \dfrac{2k+1}{2^{j-1}} - x \right) =   -f_{j,k} (x) , \quad  x  \in I_{j,k}.
\end{equation}
\end{theorem}

\begin{theorem}[Haar wavelet generation  ]
  Let $\mathcal{P}_0$ denote the one-dimensional subspace spanned by  the scaling function $ \phi (x) = \chi_{[0,1)}(x)$. Define recursively,

  \begin{equation}\label{eq:Asubj}
    \mathcal{A}_j:  \left \{   f_j \in L^2[0,1]:   f_{j,k}\left( \dfrac{2k+1}{2^{j-1}} - x \right) =   -f_{j,k} (x) , \quad  \text{ for a.e. } x  \in I_{j,k},\, k = 0,1,\cdots, 2^{j-1}-1              \right\}
  \end{equation}
  and let $ \mathcal{P}_j = \mathcal{P}_{j-1} \oplus \mathcal{A}_{j}    $. Then
  \begin{itemize}
    \item $ \mathcal{A}_j $ is the $j$-th wavelet subspace, spanned by the normalized Haar wavelet $\psi_{j-1,k}  k = 0,1,\cdots, 2^{j-1}-1  $
    \item Consequently, $\bigoplus_{j=1} ^\infty \mathcal{A}_j  $ is dense in $L^2[0,1]$.
     This construction reproduces the classical Haar wavelet decomposition and is naturally related to the Haar multiresolution analysis.
  \end{itemize}
\end{theorem}

\begin{proof}

The function $f_{j,k}$ defined in \eqref{eq:fsubjkdefined} is supported on $I_{j,k}$, takes the value $+1$ on the left half of $I_{j,k}$, and the value $-1$ on the right half. Consequently,
\begin{equation}
f_{j,k} = 2^{-(j-1)/2} \psi_{j-1,k}.
\end{equation}
Hence
\begin{equation}
\mathcal A_j = \operatorname{span} \{\psi_{j-1,k} : k=0,\dots,2^{j-1}-1 \},
\end{equation}
which is precisely the Haar wavelet subspace at scale $j-1$.
\end{proof}

\section{Conclusions and possible further Works}

In this paper, we have examined  the decomposition of a periodic function of period  $p$ into a periodic function of $p/2$ and an antiperiodic function of antiperiodic $p/2$. The newly introduced periodic function resulting from this decomposition is referred to as the first periodic generation. Continuing this process with the first periodic generation yields the second periodic generation.

If the magnitude of the $n$-th periodic generation of a given periodic function tends to zero uniformly on an initial interval $ [x_0, x_0+p] $, where $p$, is the fundamental period of $f$, then $f$ can be represented as an infinite sum of antiperiodic functions with varying fundamental antiperiods.
        $$ f(x)= c + \sum_{j=1}^{\infty}  \tilde{f}_j(x),  $$
where  $ \tilde{f}_j(x) $  is the $j$ the  antiperiodic generation of $f$ satisfying the condition $ \tilde{f}_j(x + T/2^j ) = - \tilde{f}_j(x) $. As a specific example of such indefinite decomposition we have shown that the fractional part  function $\{ x \}$ yields the  Rademacher system which is an orthogonal system that is incomplete.
Another decomposition of  the Hilbert space of integrable function $L^2[0,1]$ is decomposed into $ L^2[0,1] $.
          $$  L^2[0,1]= \mathcal{P}_0 \oplus \bigoplus_{n=1}^{\infty} \mathcal{A}_n, $$
where
$$ \mathcal{P}_0 = \operatorname{span}\{\phi\} \quad (\text{constants}),\qquad \mathcal{A}_n = \operatorname{span}\{\psi_{n,k} : k = 0,1,\dots,2^n-1\}. $$
Although the current  work is primarily emerged by analyzing solutions some difference equations, the decomposition of a space of $ T$-periodic functions into subspaces of $ T/2 $-periodic and $ T/2$ -antiperiodic functions (i.e., $ f(x+T/2) = -f(x) $) is a fundamental symmetry technique with wide-ranging applications:

\begin{itemize}
\item In Floquet theory for time-periodic systems, for a Hamiltonian $  H(t+T)=H(t) $, the time-evolution operator's monodromy matrix has eigenvalues $ e^{-i\varepsilon T} $. The quasi-energies $ \varepsilon $ are defined modulo $ 2\pi/T $. The subspaces corresponding to  $ \varepsilon = 0 $ (periodic) and $\varepsilon=\pi/ T $ (antiperiodic) play a special role in determining stability, band gaps, and the Floquet spectrum.

\item In Bloch theory in solid state physics, in a periodic potential with lattice constant $a$, Bloch functions are classified by crystal momentum $k$. The decomposition into periodic ($ k=0 $ ) and antiperiodic ($ k=\pi/a $) boundary conditions occurs at the center and edge of the Brillouin zone. This helps analyze energy band structures, especially band crossings and topological properties.

\item  Eigenfunctions of periodic differential operators (e.g., Mathieu equation) are either periodic or antiperiodic with half the period. The decomposition yields two separate eigenvalue sequences, which are used to determine stability regions (Ince--Strutt diagrams) and to construct Green's functions.

\item  In Signal processing and harmonic analysis, Any $T$-periodic signal can be split into components with half-period symmetry. This reduces data storage (e.g., compressing symmetric waveforms) and simplifies filtering operations, especially in systems where half-period delays are natural (e.g., push-pull amplifiers).

\item   For PDEs with periodic boundary conditions, using basis functions that are a priori periodic or antiperiodic (e.g., even/odd Fourier series) decouples the problem into independent subsystems, reducing computational cost and improving accuracy.

\end{itemize}

\section*{Statements of Declarations}

\subsection*{Conflict of interests}

The author declares that there is no conflict of interests regarding the publication of this paper.

\subsection{Originality of the work}

The main results presented in this paper are original to the author, constitute new findings, and have not been published elsewhere.

\subsection*{Author contribution}
The corresponding author is the sole originator and author of this article.

\subsection*{Data availability}

There are no external data sources used in this paper other than the reference materials cited within it.

\subsection*{funding}

This research work has not received funding from any organization or individual.

\subsection*{Acknowledgment}

The author is grateful to AAU members Prof. Tadesse Abdi and Prof. Samuel Asefa for their valuable comments on the final draft. The e-print version of this manuscript by the current author is available at \cite{HBY}.


\end{document}